\documentclass[11pt]{amsart}

\usepackage{amsthm,amssymb,amsmath,hyperref}
\usepackage{color}
\usepackage{tikz}
\usepackage{float}
\usepackage[T5,T1]{fontenc}
	\usepackage{multicol}
\usepackage{fullpage} 

\hypersetup{
	colorlinks=true,
	citecolor=[rgb]{0,0,0.5},
	linkcolor=[rgb]{0,0,0.5},
	urlcolor=[rgb]{0,0,0.75},
	pdfpagemode=UseNone,
	pdfstartview=FitH,
	pdfdisplaydoctitle=true,
	pdftitle={Sparse bounds for discrete singular Radon transforms},
	pdfauthor={T. C. Anderson, B. Hu, J. Roos},
	pdflang=en-US
}

\theoremstyle{plain}
\newtheorem{thm}{Theorem}[section]
\newtheorem{thma}{Theorem}
\newtheorem{lem}[thm]{Lemma}
\newtheorem{cor}[thma]{Corollary}
\newtheorem{prop}[thm]{Proposition}
\theoremstyle{definition}

\newtheorem*{defna}{Definition}
\theoremstyle{remark}

\def\R{\mathbb{R}}
\def\Z{\mathbb{Z}}
\def\C{\mathbb{C}}
\def\Q{\mathbb{Q}}
\def\N{\mathbb{N}}

\begin{document}
	
	\title[Sparse bounds]{Sparse bounds for discrete singular Radon transforms}
	\author{Theresa C. Anderson\quad Bingyang Hu\quad Joris Roos}

	\address{Theresa C. Anderson \\ Purdue University \\150 N University St, West Lafayette, IN 47907, USA} \email{tcanderson@purdue.edu}

	\address{Bingyang Hu \\ Purdue University \\150 N University St, West Lafayette, IN 47907, USA}
	\email{hu776@purdue.edu}
	
    \address{Joris Roos \\ University of Massachusetts Lowell \\
    220 Pawtucket St, Lowell, MA 01854, USA}
    \email{jroos.math@gmail.com}
	
	\begin{abstract}
		We show that discrete singular Radon transforms along a certain class of polynomial mappings $P:\mathbb{Z}^d\to \mathbb{Z}^n$ satisfy sparse bounds. For $n=d=1$ we can handle all polynomials. In higher dimensions, we pose restrictions on the admissible polynomial mappings stemming from a combination of interacting geometric, analytic and number-theoretic obstacles. 
	\end{abstract}
	\date{August 4, 2020}
	\maketitle
	
	\section{Introduction}
	Let us consider the operator
	\begin{equation}\nonumber 
	T_P f(x) = \sum_{y\in\Z^d\setminus \{0\} } f(x+P(y)) K(y), \quad (x\in\Z^n)
	\end{equation}
	acting on functions $f:\Z^n\to \C$, where $P:\Z^d\to \Z^n$ is a polynomial mapping and $K$ a Calder\'{o}n--Zygmund kernel in $\R^d$ satisfying
	\begin{equation}\nonumber 
	|y|^d |K(y)|+|y|^{d+1} |\nabla K(y)| \le 1\quad \text{for all}\; |y| \ge 1
	\end{equation}
	and
	$$
	\sup_{\lambda \ge 1} \left| \int_{1 \le |y| \le \lambda} K(y) dy \right| \le 1.
	$$
	
	Ionescu and Wainger \cite{IW06} proved that $T_P$ is bounded $\ell^p(\Z^n)\to \ell^p(\Z^n)$ for all $p\in (1,\infty)$ (also see recent work of Mirek \cite{Mir15}). This extended earlier work of Stein and Wainger \cite{SW99}. Maximal truncations, vector-valued estimates and variation-norm estimates associated with $T_P$ have recently been studied by Mirek, Stein and Trojan \cite{MST15a}, \cite{MST15b}, introducing a number of key new ideas.  See also \cite{MSZ} for extensions beyond the Calder\'on-Zygmund setting.
		
	In this paper we consider sparse bounds for the operator $T_P$. Sparse bounds go back to Lerner's alternative proof of the $A_2$ theorem \cite{L13} and there have since been a number of important further developments. Here we use the concept of \emph{bilinear} sparse bounds from \cite{BFP}, \cite{CDO18}, \cite{L17}. 
	Recently, many new directions have been studied, including sparse domination of Radon transforms \cite{CO17}, \cite{O19}, \cite{H19}, discrete analogues, \cite{CKL16}, \cite{KL18} and operators on spaces of homogeneous type \cite{AV14}.
	
	As far as we are aware sparse bounds for our operator $T_P$ are only known in the case $n=d=1$, $P(y)=y^3$, $K(y)=1/y$, due to Culiuc-Kesler-Lacey \cite{CKL16}. The proof of \cite{CKL16} also works for the case $P(y)=y^k$ with no modifications to the argument if $k$ is odd (if $k$ is even, the corresponding operator vanishes identically).
	Our goal in this paper is to push the methods of \cite{CKL16} to prove sparse bounds for the operator $T_P$ for a large class of polynomial mappings $P$. Let us write
	\[ P(t) = \sum_{\alpha} c_\alpha t^\alpha, \]
	where the sum is over multiindices $\alpha\in\N_0^d$ and only finitely many of the $(c_\alpha)_\alpha$ are non-zero.
	
	\begin{samepage}
	\begin{defna} \label{admissible}
		We say that $P=(P_1,\dots,P_n)$ is \emph{admissible} if $P(0)=0$, $c_\alpha\in\Z^n$ for all $\alpha$ and the following conditions are satisfied:
		\begin{enumerate}
			\item [1.] \emph{Condition (\L).}
			There exist constants $\beta, L_0>0$ such that $|P(t)| \ge |t|^\beta$ for all $|t| \ge L_0$. 
			\item [2.] \emph{Condition (C).} For each $i \in \{1, \dots, n\}$ there exists a multiindex $\alpha^{(i)}$ with $|\alpha^{(i)}| = \deg P_i$ such that $c_\alpha$ equals the $i$th unit vector in $\R^n$.
		\end{enumerate}
	\end{defna}
	\end{samepage}
	We note that Condition (C) is always satisfied if $n=1$ (up to our normalizing assumptions that the leading coefficient of $P$ equals one and the constant term equals zero). Condition (\L) on the other hand holds always if $d=1$. In particular, the class of admissible polynomials includes all monic polynomials without constant term if $n=d=1$.
	
	Other examples of admissible $P$ include the monomial curves $(t,t^k)$ with $k\ge 1$ and the moment curve $(t,t^2,\dots,t^k)$. Another example is the ``universal'' case from Ionescu-Wainger \cite{IW06}, given by $[P(t)]_\alpha = t^\alpha$, where $\alpha$ ranges over all multiindices $\alpha\in\N_0^d$ with $0<|\alpha|\le D$, where $D$ is fixed. Ionescu and Wainger use the \emph{method of descent} \cite[Ch. XI]{Ste93} to reduce the case of a general $P$ to this case. It appears that the method of descent is not applicable to the sparse bounds we consider.

	We require Condition (\L) in order to ensure a lower bound on the sidelength of cubes, which also corresponds to a crucial step in the argument of \cite{CKL16}. Condition (\L) is related to {\L}ojasiewicz-type inequalities, which is a classical topic in real algebraic geometry. These are inequalities relating $|P(t)|$ to the distance of $t$ to the zero set $Z(P)=\{t\in\R^d\,:\,P(t)=0\}$. The type of {\L}ojasiewicz inequality that Condition (\L) represents has been shown to hold for various classes of real polynomials in \cite{DHPT14}.
	Note that Condition (\L) implies that $Z(P)$ is bounded. However, the reverse is not true: if $P$ has a bounded zero set, then it may still happen that $|P(t)|$ is small when $|t|$ is large. Consider for example, $P(t)=t_1^2 + t_2^2 (1-t_1t_2)^2$. Then $Z(P)=\{0\}$, but $|P(t)|\le 2$ on the unbounded set $\{t_1t_2=1,\,|t_1|\le 1\}$. 
	
	Condition (C) enters the analysis in several ways: it allows us to easily deduce the required decay for an exponential sum appearing in the major arc term and is also crucial for the proof of the error estimates that enable the analysis of the minor arc term.  Moreover, it implies that the associated map $\gamma(x,t) = x-P(t)$ satisfies H\"ormander's condition, i.e. that the associated (constant coefficient) vector fields span $\R^n$. This is needed for the application of sparse bounds for the real-variable operator from the second author's thesis \cite{H19}.

	To state our result we first introduce some terminology. We call a set $Q=\Z^n\cap (I_1\times \cdots\times I_n)$ with $I_1,\dots,I_n\subset\R$ intervals a \emph{$P$-cube} if
	\[ |I_1|^{1/D_1} = \cdots = |I_n|^{1/D_n}, \]
	where $D_i=\mathrm{deg}\,P_i$ and $|I_l|$ is the length of the interval $I_l$. In that case, we refer to $|I_1\cap \Z|^{1/D_1}$ as the {\em sidelength} of $Q$, denoted $\ell(Q)$. Note that $\ell(Q)\approx |I_i\cap\mathbb{Z}|^{1/D_i}$ for all $i=1,\dots,n$.  Here $|\cdot|$ denotes the counting measure.

	For fixed $\sigma\in(0,1)$ we say that a family $\mathcal{S}$ of $P$-cubes in $\Z^n$ is \emph{$\sigma$-sparse} if for every $Q\in\mathcal{S}$ there exists a set $E_Q\subset Q$ such that $|E_Q|\ge \sigma |Q|$ and the sets $(E_Q)_{Q\in\mathcal{S}}$ are pairwise disjoint. Given a sparse collection $\mathcal{S}$ we define an associated \emph{sparse form} $\Lambda_{r,s}^\mathcal{S}$ by
	\[ \Lambda^{{\mathcal{S}}}_{r, s}(f, g) = \sum_{Q\in\mathcal{S}} |Q| \langle f\rangle_{Q, r} \langle g\rangle_{Q, s} \]
	with $r,s\in [1,\infty]$, $\langle f\rangle_{Q, r} = \big(|Q|^{-1}\sum\limits_{x\in Q} |f(x)|^r\big)^{1/r}$ for $r<\infty$ and $\langle f\rangle_{Q,\infty} = \sup_{x\in Q} |f(x)|$. We may write $\Lambda_{r,s}$ instead of $\Lambda_{r,s}^\mathcal{S}$ when $\mathcal{S}$ is clear from context. Note that by H\"older's inequality we have $\Lambda_{r,s}(f,g)\le \Lambda_{\rho,\iota}(f,g)$ for every $\rho\in [r,\infty], \iota\in [s,\infty]$.
	
	The following is our main result.
	
	\begin{samepage}
	\begin{thma}\label{mainresult}
		Let $P:\Z^d\to \Z^n$ be admissible, $(\frac1r,\frac1s)$ sufficiently close to $(\frac12,\frac12)$ and $\sigma\in(0,1)$. Then there exists a constant $C=C(P,r,s,\sigma)$ such that for every finitely supported $f,g:\Z^n\to\C$ there exists a $\sigma$--sparse collection $\mathcal{S}$ in $\Z^n$ with
		\[ |\langle T_P f,g\rangle| \le C \Lambda^\mathcal{S}_{r,s}(f,g). \]
	\end{thma}
	\end{samepage}
	
	\noindent {\emph{Remark.} The result directly implies certain weighted $L^p$ estimates for the operator $T_P$ with weights in the intersection of an appropriate Muckenhoupt $A_p$ class and a reverse H\"older class $\mathrm{RH}_s$, because such bounds are known to hold for the bisublinear forms $\Lambda^\mathcal{S}_{r,s}$. This was proved by Bernicot, Frey and Petermichl in \cite[\S 6]{BFP}.  
	More specifically (choosing $r = s$ for simplicity), we have for instance:
	\begin{cor}
	If $r$ is such that $2-r>0$ is sufficiently small and $w, w^{-1} \in A_{\frac2r}\cap \mathrm{RH}_{\frac{r}{2-r}}$, then $T_P$ is bounded on $\ell^2(w)$.
	\end{cor}
    
	As far as we are aware, no non-trivial weighted bounds for $T_P$ were known previously other than for the case considered in \cite{CKL16}.}
	
	For the proof we use an appropriate major-minor arc decomposition of the Fourier multiplier associated with $T_P$. The major arc component encodes both arithmetic (discrete) and analytic (continuous) information. We closely follow the arguments of Culiuc, Kesler and Lacey \cite{CKL16} and aim to extend the arguments to the largest possible class of polynomial mappings.
	One difference is that our analysis is adapted to the anisotropic geometry of the map $P$; this  complication was not present in \cite{CKL16}.
	Another new difficulty when compared to \cite{CKL16} is that the corresponding  real-variable operator is no longer of Calder\'{o}n-Zygmund type. Instead, we need to make use of the recently developed general theory for sparse domination of Radon transforms on $\R^n$, see \cite{H19} (also see \cite{CO17} for an earlier result). 
	To this end, we also remark that the {\em pointwise} sparse domination paradigm of \cite{L13} does not apply.
	
	A difficult open question is to determine the sharp range of $r,s$ for which a sparse domination result as in Theorem \ref{mainresult} can hold, even in the simplest case of $n=d=1$, $P(t)=t^3$, $K(y)=1/y$.
	We stress that the present range of $(r,s)$ is not sufficient to recover the full $\ell^p$ boundedness result of Ionescu and Wainger.  One can calculate the range of $p$ that our result does give by viewing the last section of the paper where we make our range of $r$ and $s$ quantitative, with $r = p$ (and $s = p$) in the language of Ionescu and Wainger.  The explicit range of $p$ can be computed for certain operators, but not (currently) for all that we consider.  

	The range of $r,s$ that is produced by our proof is determined by three corresponding parts of the argument. The first restriction is from the known range for the corresponding continuous (real--variable) Radon transform (which is determined by its $L^p$-improving range, see \S \ref{sec:realvar}). Here the sharp range is known in some particular cases, but not in the generality we consider. The second restriction comes from the decay rate of a certain Weyl sum. Finally, the third input stems from the minor arc error term. The region of $(\frac1r,\frac1s)$ we obtain is explained in more detail in \S \ref{sec:quant}.
	
	In our definition of admissibility we have traded some generality for the sake of simplifying the argument. However, we presently do not see how to obtain sparse bounds for {\em every} polynomial mapping $P$. This could be an interesting topic for future investigations.

	\noindent \emph{Structure of this paper.} In \S \ref{sec:prelim} we introduce some notation and propositions that we will use; this includes an estimate of the exponential sums that arise in our analysis using Condition (C). In \S \ref{sec:approx} we perform the decomposition of our operator into major and minor arc terms and explain how Theorem \ref{mainresult} is proven. In \S \ref{sec:realvar} we use Condition (\L) to derive a variant of the real--variable sparse bound proven in \cite{H19}. This is a key ingredient for the major arc estimate.
	In \S \ref{sec:major} and \S \ref{sec:minor} we deal with the major and minor arc components of our operator, respectively. Finally, \S \ref{sec:quant} contains a more quantitative version of Theorem \ref{mainresult}.
	
	\noindent {\em Acknowledgements.} This research is supported by the NSF, in particular by NSF DMS-1502464 and NSF DMS 1954407 (T. C. Anderson). We thank the anonymous referees for helpful comments.
	
	\section{Preliminaries}\label{sec:prelim}
	\subsection{Notation}
	
	Throughout the text we make use of the notation $A\lesssim B$ to denote that there exists a constant $C>0$ such that $A\le C\cdot B$, where $C$ is always allowed to depend on the polynomial mapping $P$ (in particular on $d$ and $n$), but is independent of the coefficients $c_\alpha$.
	Let $e(t)=e^{2\pi i t}$. We say that a function $g:\R^n\to\C$ is {\em periodic} if $g(x+y)=g(x)$ for all $x\in\R^n$ and $y\in\Z^n$.
	For $f:\Z^n\to \C$ and periodic $g:\R^n\to\C$,
	\[ \mathcal{F} f(\xi) =\widehat{f}(\xi) = \sum_{x\in\Z^n} f(x) e(-\xi x)\quad (\xi\in\R^n) \]
	\[ \mathcal{F}^{-1} g(x) = \int_{[0,1]^n} g(\xi) e(\xi x) d\xi \quad (x\in\Z^n)\]
	We also fix a notation for the Fourier transform on $\R^n$,
	\[ \mathcal{F}_{\R^n} f(\xi) = \int_{\R^n} f(t) e(-\xi\cdot t) dt. \]
	For a positive integer $q$ and $a\in \Z^n$ we write
	\[ [q] = \{0,\dots, q-1\},\quad (a,q) = \mathrm{gcd}(a_1,a_2,\dots,a_n,q). \]
	By convention, the letter $q$ will always denote a positive integer and  $a/q\in\Q^n$ is always assumed to satisfy $(a,q)=1$.
	
	\subsection{Dyadic decomposition}
	Fix a smooth function $\psi$ supported in $\{1/2\le |x|\le 2\}\subset \R^d$ such that $0\le \psi\le 1$ and $\sum_{j\in\Z} \psi(2^{-j} x)=1$ for every $x\not=0$. Define
	\[ K_j(x) = \psi(2^{-j} x) K(x), \quad \Phi_j(\xi) = \int_{\R^d} e(P(t)\cdot \xi) K_j(t) dt, \]
	\[ m_j(\xi) = \sum_{y\in\Z^d} e(P(y)\cdot\xi) K_j(y). \]
	For a generic bounded periodic function $m:\R^n\to \C$ let us write $m(\nabla)$ to denote the corresponding Fourier multiplier operator acting on functions $f:\Z^n\to \C$ by
	\[\mathcal{F} (m(\nabla) f)(\xi) = m(\xi) \mathcal{F}(f)(\xi).\]
	
	\subsection{An exponential sum estimate}	
	Define for every $(a,q)=1$, the Weyl-type exponential sum
	
	\[ S(a/q) = q^{-d} \sum_{r\in [q]^d} e(P(r)\cdot a/q). \]
	
	We claim that for every $\varepsilon>0$,
	\begin{equation} \label{20190805eq04} 
	|S(a/q)| \lesssim_\varepsilon q^{-1/D_*+\varepsilon},
	\end{equation} 
	where $D_*=\max\{ \alpha_i\,:\,c_\alpha\not=0,\,i=1,\dots,d\}\le \deg\,P$.
	To see this note that the phase in the exponential sum is given by
	\[ P(r)\cdot a/q=\sum_{\alpha} \tfrac{b_\alpha}{q} r^\alpha, \]
	where we put $b_\alpha = c_\alpha\cdot a$. Condition (C) implies that the coefficient of $r^{\alpha^{(i)}}$ is $b_{\alpha^{(i)}}=a_i/q$; since $(a,q)=1$ it follows that $(b,q)=1$. A standard exponential sum estimate (more precisely, \cite[Theorem 2.6]{ACK}) then implies the claim.
	From an analytic number theory perspective, there are other ways to expand Condition (C) to possibly guarantee the exponential sum decay, but Condition (C) also appears in other places of the proof.

	\subsection{Sparse forms}
	
	Recall the following fundamental fact about sparse forms.
	
	\begin{lem}\label{lem:uniformsparse}
		Let $\sigma\in (0,1)$ and $1\le r,s<\infty$. There exists a constant $C>0$ and a $\sigma$-sparse form $\Lambda_{r,s}^*$ such that for every $\sigma$-sparse form $\Lambda_{r,s}$ we have
		\[ \Lambda_{r,s}(f,g) \le C \Lambda^*_{r,s}(f,g) \]
		for all compactly supported $f,g:\Z^n\to\C$.
	\end{lem}
	This is a variant of \cite[Lemma 4.7]{LA17} and the proof given there carries over.
	
	For $r,s\in [1,\infty]$ and $T$ a sublinear operator mapping finitely supported functions on $\Z^n$ to functions on $\Z^n$, define the \emph{sparse operator norm $\|T\|_{\mathrm{sp}(r,s)}$} as the infimum over all $c\in (0,\infty)$ such that 
	\[ |\langle Tf,g\rangle_{\ell^2(\Z^n)}| \le c \sup \Lambda_{r,s}(f,g) \]
	holds for all compactly supported $f,g:\Z^n\to\C$, where the supremum is over all sparse forms $\Lambda_{r,s}$ (which is finite by Lemma \ref{lem:uniformsparse}). Note that this is with respect to a fixed sparsity parameter $\sigma\in (0,1)$ which does not appear in notation.	
	
	We define the maximal operator with respect to $P$-cubes as
	$$
	M_P f(x)=\sup_{Q} |Q|^{-1} \sum_{y \in Q} |f(x-y)|,
	$$
	where the supremum goes over all $P$-cubes $Q$.
	
	For a $P$-cube $Q$ we let $2Q$ denotes the smallest $P$-cube with the same center as $Q$ and sidelength at least twice the sidelength of $Q$. We have $|Q|\approx \ell(Q)^{D_1+\cdots+D_n}$, where $\ell(Q)$ denotes sidelength of $Q$ and $D_i=\mathrm{deg}\,P_i$. For $x\in Q$ and $y\not\in 2Q$ we have $|x-y|\gtrsim \ell(Q)$. The $P$-cube $2^\nu Q$ is defined inductively by $2^{\nu} Q= 2 (2^{\nu-1} Q)$.
	
	\begin{prop} \label{sparseHL}
		The operator $M_P$ satisfies
		$$
		\|M_P\|_{\mathrm{sp}(1,1)} \lesssim 1.
		$$
	\end{prop} 
	This is a variant of the standard sparse domination theorem for the Hardy-Littlewood maximal operator. It follows for example as a special case of the pointwise sparse domination theorem for maximal operators in spaces of homogeneous type, see \cite[Chapter 3]{A15}.
	
	We also need the following variant of \cite[Proposition 2.4]{CKL16}.
	\begin{prop} \label{prop:Kfinitesupport}
		Suppose that $K:\Z^n\to \C$ is supported on a $P$-cube $Q_*$ centered at the origin and write $T_K f=K*f$. Then for all $(\frac1r,\frac1s)\in [0,1]^2$,
		$$
		\|T_K\|_{\mathrm{sp}(r, s)} \lesssim_n   |Q_*|^{\tfrac{1}{r}+\tfrac{1}{s}-1}   \|T_K\|_{\ell^r \to \ell^{s'}}. $$
	\end{prop}

	\begin{proof}
		Let $\mathcal{S}$ be a finitely overlapping partition of $\Z^n$ consisting of translates of $Q_*$. Then
		\[ |\langle T_K f, g\rangle|  \lesssim \sum_{Q\in\mathcal{S}} \big|\sum_{x\in Q} \sum_{y\in\Z^n} K(x-y) f(y) g(x)\big|.\]
		Since $K$ is supported on $Q_*$ we have $y\in Q-Q_*$ if the summand is non-zero, so the previous equals
		\[ \sum_{Q\in\mathcal{S}} |\langle T_K (f\mathbf{1}_{Q-Q_*}), g\mathbf{1}_{Q}\rangle| \le \|T_{K}\|_{\ell^r\to \ell^{s'}} \sum_{Q\in\mathcal{S}} \|f\mathbf{1}_{Q-Q_*}\|_{r} \|g\mathbf{1}_Q\|_s, \]
		where we have used H\"older's inequality to estimate the inner product and the definition of the operator norm $\|T_K\|_{\ell^r\to \ell^{s'}}$.
		Since $Q_*$ is centered at the origin and $Q$ and $Q_*$ are $P$-cubes of equal sidelength we have $Q-Q_*\subset 4Q$. Hence the previous display is
		\[ \le \|T_K\|_{\ell^r\to \ell^{s'}} \sum_{Q\in\mathcal{S}} |4Q|^{1/r} |Q|^{1/s} \langle f\rangle_{4Q,r} \langle g\rangle_{Q,s}, \]
		which is
		\[ \lesssim \|T_K\|_{\ell^r\to \ell^{s'}} |Q_*|^{\frac1r+\frac1s-1} \sum_{Q\in\mathcal{S}} |Q| \langle f\rangle_{4Q,r} \langle g\rangle_{4Q,s}.\]
		Since $4\mathcal{S}=\{4Q\,:\,Q\in\mathcal{S}\}$ is a sparse collection (with adjustable sparsity by writing it as a finite union of sparser collections), we obtain the claim.
	\end{proof}

	\section{Multiplier approximations}\label{sec:approx}
	\label{expsum:section}
	The first step towards Theorem \ref{mainresult} is an approximation for the multipliers $m_j(\xi)$. 
	
	Let $\delta\in (0,\tfrac1{100})$ be a sufficiently small parameter and define the {\em major arc} corresponding to $a/q\in\Q^n$ by
	\begin{equation} \label{20190921eq03}
	\mathfrak{M}_j(a/q) = \left\{ \xi\in\R^n\,:\,|\xi_i-a_i/q|\le 2^{-(D_i-1)j} 2^{-\delta j}\;\text{for all}\;1\le i \le n \right\},
	\end{equation}
	where $D_i=\deg\,P_i$. We need another parameter that governs the size of denominators. Say that $\delta'\in (0,\tfrac1{10} \delta)$ and
	\[ \mathfrak{M}_j = \bigcup_{(a,q)=1,\,q\le 2^{\delta' j}} \mathfrak{M}_j(a/q),  \quad j \ge 1. \]
	The major arcs appearing in this union are pairwise disjoint. 
	
	\begin{lem}\label{lem:approx}
		For any $j \ge 1$, if $\xi\in\mathfrak{M}_j(a/q)$, $q\le 2^{\delta'j}$ then
		\[ m_j(\xi) = S(a/q)\Phi_j(\xi-a/q) + O(2^{-\varepsilon j}), \]
		where $\varepsilon=\delta-\delta'\in (0,1)$.
	\end{lem}
	In its essence this lemma goes back to Bourgain \cite{Bourgain}. We postpone the standard proof of this to the end of this section. 
	
	A key step is now to dyadically group rationals according to the size of their denominator. Define for every integer $k\ge 1$,
	\[ \mathcal{R}_k = \{ a/q \in \Q^n\cap [0,1)^n\,:\,(a,q)=1, 2^{k-1}\le q< 2^k \}. \]
	Let $\chi$ be a smooth function supported on $\{\xi\in\R^n\,:\,|\xi|\le 2\}$ that equals one on $\{\xi\in\R^n\,:\,|\xi|\le 1\}$.
	Define for $j\ge 1$ and $k\ge 1$ the periodic function,
	\[ L_{j,k}(\xi) = \sum_{a/q\in\mathcal{R}_k} S(a/q) \Phi_j(\xi-a/q) \chi_k(\xi-a/q), \]
	where $\chi_k$ is the periodic function with $\chi_k(\xi)=\chi(2^{10 k} \xi )$ for all $\xi\in [0,1)^n$. For every $\xi\in\R^n$ at most one of the summands is non-zero (if two different $a/q$ were both in the support of $\chi_k$, then we get a contradiction since $q \in [2^{k-1}, 2^k]$). Moreover, the major arcs $\mathfrak{M}_j(a/q)$ for the rationals $a/q\in\mathcal{R}_k$ are comfortably contained in the support of $L_{j,k}$. More precisely, for $\xi~\in~\mathfrak{M}_j(a_*/q_*)$ and $a_*/q_*\in\mathcal{R}_{k_*}$ we have
	\[ L_{j,k_*}(\xi) = S(a_*/q_*) \Phi_j(\xi-a_*/q_*). \]
	Next let
	\[ L_{j}= \sum_{1\le k\le j\delta'} L_{j, k},\quad L^{(k)} = \sum_{j\ge k/\delta'} L_{j,k}. \]
	Observe that the supports of $L_{j,k}$ for different $k$ may overlap. 	Motivated by Lemma \ref{lem:approx} write
	\[ m_j = L_j + E_j \]
	and we take this as the definition of the error term $E_j$.
	With this in mind, Theorem \ref{mainresult} follows from the following two results.
	
	\begin{prop}\label{prop:main} Suppose $(\frac1r,\frac1s)$ is sufficiently close to $(\frac12,\frac12)$. Then there exists $\gamma>0$ such that
		\[ \|L^{(k)} (\nabla)\|_{\mathrm{sp}(r,s)} \lesssim 2^{-\gamma k} \]
		for all $k\ge 1$.
	\end{prop}
	The proof of this constitutes the core of the paper and is contained in \S \ref{sec:major}.

	\begin{prop}\label{prop:err} Suppose $(\frac1r,\frac1s)$ is sufficiently close to $(\frac12,\frac12)$. Then there exists $\gamma>0$ such that
		\[ \|E_j(\nabla)\|_{\mathrm{sp}(r,s)} \lesssim 2^{-\gamma j} \]
		for all $j\ge 1$.
	\end{prop}
	The proof of this proposition is contained in \S \ref{sec:minor}. Theorem \ref{mainresult} now follows because for $(\frac1r,\frac1s)$ in the intersection of the two regions in Propositions \ref{prop:main} and \ref{prop:err} we have
	\[ \|\sum_{j\ge 1} m_j(\nabla)\|_{\mathrm{sp}(r,s)} \le \sum_{k\ge 1} \| L^{(k)}(\nabla)\|_{\mathrm{sp}(r,s)} + \sum_{j\ge 1} \|E_j(\nabla)\|_{\mathrm{sp}(r,s)}<\infty.\]

	\begin{proof}[Proof of Lemma \ref{lem:approx}]
		For every $y\in\Z^d$ there exist unique $u\in\Z^d, r\in [q]^d$ such that $y=uq+r$. Thus, 
		\[ m_j(\xi) = \sum_{y\in\Z^d} e(P(y)\cdot a/q) e(P(y)\cdot (\xi-a/q)) K_j(y)= q^{-d} \sum_{r\in [q]^d} e(P(r)\cdot a/q) I_{q,r}(\xi), \]
		where
		\[ I_{q,r}(\xi) = q^d \sum_{u\in \Z^d} e(P(uq+r)\cdot (\xi-a/q)) K_j(uq+r). \]
		It suffices to show that for $r\in [q]^d$,
		\[ I_{q,r}(\xi) = \Phi_j(\xi-a/q) + O(2^{-\varepsilon j}). \]
		Write $\eta=\xi-a/q$. A change of variables gives
		\[ \Phi_j(\eta) = \int_{\R^d} e(P(t)\cdot \eta) K_j(t) dt = q^d \int_{\R^d} e(P(qt+r)\cdot \eta) K_j(qt+r) dt,  \]
		which can be further rewritten as
		\begin{equation}\label{eqn:approx_pf0}
		q^d \sum_{u\in\Z^d} \int_{[0,1]^d} e(P(qu+r+qt)\cdot \eta) K_j(qu+r+qt) dt 
		\end{equation} 
		For every $i=1,\dots,n$ we obtain from the mean value theorem,
		\[ |(P_i(qu+r+qt) - P_i(qu+r)) \eta_i| \le q |\eta| \sup_{|z|\approx 2^j} |P_i'(z)| \lesssim 2^{\delta' j} 2^{-(D_i-1)j} 2^{-\delta j} 2^{(D_i-1)j} = 2^{-\varepsilon j}, \]
		where we have used that $|qu+r+qt|$ and $|uq+r|$ are $\approx 2^j$ and that $t\in [0,1]^d$, $q\le 2^{\delta' j}$, $|\eta_i|\le 2^{-(D_i-1)j} 2^{-\delta j}$.
		Since $\int_{\R^d} |K_j(t)| dt\approx 1$ this implies that \eqref{eqn:approx_pf0} is
		\begin{equation}\label{eqn:approx_pf1}
		q^d \sum_{u\in\Z^d} e(P(qu+r)\cdot \eta) \int_{[0,1]^d} K_j(qu+r+qt) dt  + O(2^{-j\varepsilon})
		\end{equation}
		Again from the mean value theorem we see
		\[ |K_j(qu+r+qt) - K_j(qu+r)| \le q \sup_{|z|\approx 2^j} |\nabla K_j(z)| \lesssim 2^{\delta' j} 2^{-(d+1)j} \le 2^{-dj} 2^{-\varepsilon j}, \]
		using that $t\in [0,1]^d$, $|qu+r+qt|\approx |qu+r|\approx 2^j$ and $q\le 2^{\delta' j}$ and $\delta<1$. Consequently, \eqref{eqn:approx_pf1} is
		\[ q^d \sum_{u\in\Z^d} e(P(qu+r)\cdot \eta) K_j(qu+r) + O(2^{-j\varepsilon}), \]
		which establishes the claim.
	\end{proof}

	\section{Sparse bounds for a real--variable Radon transform}\label{sec:realvar}
	
	A key step in the analysis of the major arcs is the application of a vector--valued sparse bound for the real-variable singular Radon transform
	\[ \mathcal{T}_{P,K} F(x) = \int_{\R^d} F(x+P(t)) K(t) dt, \]
    with $K$ a Calder\'{o}n--Zygmund kernel.
	The operator $\mathcal{T}_{P,K}$ can be interpreted as acting on functions on Euclidean space $\R^n$, or as acting on functions on the integer lattices $\Z^n$. We may also let the operator act on functions taking values in a Hilbert space. In order to simplify notation, we use the same letter for each of these interpretations. However, it will always be unambiguous from context which interpretation is meant.
	
	In our later application we will encounter the operator $\mathcal{T}_{P,K}$ interpreted as acting on functions $F:\Z^n\to \mathcal{H}$, where $\mathcal{H}$ is a separable complex Hilbert space. 
	The notion of sparse form naturally extends to $\mathcal{H}$--valued functions: for $F,G:\Z^n\to \mathcal{H}$, $r,s\in [1,\infty)$ and a sparse collection $\mathcal{S}$ we let
	\[\Lambda^{\mathcal{S}}_{r,s}(F,G) = \sum_{Q\in\mathcal{S}} |Q| \langle F\rangle_{Q,r} \langle G\rangle_{Q,s}, \]
	where
	\[ \langle F\rangle_{Q,r} = ( |Q|^{-1} \sum_{x\in Q} \|F(x)\|_\mathcal{H}^r)^{1/r}. \]
	
	Sparse bounds for real-variable singular Radon transforms were studied in a much more general setting in \cite{H19}. For our purpose, we need the following additional information on the sparse family: roughly speaking, if the support of $K$ is far from the origin, then the cubes in the sparse collection should have large sidelength. To ensure this we will depend on Condition (\L): there exist $\beta,L_0>0$ such that $|P(t)|\ge |t|^\beta$ for $|t|\ge L_0$.
	
	Sparse bounds for $\mathcal{T}_{P,K}$ are closely related to $L^p$ improving estimates for the corresponding single scale operators (see \cite{L17}).
	Let $\Omega_c$ denote the interior of the set of all $(\frac1r,\frac1s)\in [0,1]^2$ such that there exist $p\in [1,r]$ and $q\in [s',\infty]$ with $q>p$ such that
	\[ \|\mathcal{T}_{P,\psi} f\|_{L^q(\R^n)}\lesssim \|f\|_{L^p(\R^n)} \]
	holds for all test functions $f:\R^n\to \C$ and every smooth $\psi$ supported in $\{x\in\R^n\,:\,\frac12\le |x|\le 2\}$ with constants depending only on the supremum norm of $\psi$ and of derivatives of $\psi$ up to some finite order.
	
	\begin{prop}\label{prop:Hilbertsparse}
		Let $K$ be a Calder\'{o}n-Zygmund kernel (in the sense defined in the introduction) on $\R^d$ and $P:\Z^d\to \Z^n$ be admissible. Then for compactly supported functions $F,G:\Z^n\to \mathcal{H}$, $\sigma\in(0,1)$ and $(\frac1r,\frac1s)\in \Omega_c$ there exists a $\sigma$--sparse family $\mathcal{S}$ of $P$-cubes such that
		\begin{equation} \label{eqn:realsparse}
		|\langle \mathcal{T}_{P,K} F, G\rangle| \lesssim \Lambda_{r,s}^{\mathcal{S}}(F, G),
		\end{equation}
		where the constant depends on $P,r,s,\sigma$. Moreover, if $K$ is supported on $\{|t|\ge L\}$ for some $L\ge L_0$, then $\mathcal{S}$ can be chosen such that the sidelength of each $P$-cube $Q\in\mathcal{S}$ is $\gtrsim L^{\beta/D}$, where $D=\mathrm{deg}\,P$.
	\end{prop}
	
	\begin{proof}[Proof of Proposition \ref{prop:Hilbertsparse}]
		The estimate \eqref{eqn:realsparse} is a consequence of \cite[Theorem 1.4]{H19}, which we apply after extending $F$ and $G$ to $\R^n$ as follows: for $x\in\R^n$ define $F(x)=F(y)$ where $y\in\Z^n$ arises from rounding the components of $x$ to nearest integers. Condition (C) ensures that \cite[Theorem 1.4]{H19} is applicable and that the cubes in the resulting sparse collection may be taken to be $P$--cubes. 
		
		The second claim follows from the proof of \cite[Theorem 1.4]{H19}. A key step is the estimate
		$$
		|\langle \mathcal{T}_{P,K} (F\mathbf{1}_Q), G\mathbf{1}_Q\rangle| \lesssim |Q| \langle F \rangle_{Q, r}\langle G \rangle_{Q, s}
		$$
		for those cubes $Q$ that are selected in the sparse collection
		(see \cite[bottom of p. 284 and (6.18)]{H19}). If $\langle \mathcal{T}_{P,K} (F\mathbf{1}_Q), G\mathbf{1}_Q\rangle=0$ then the cube $Q$ can be omitted from the sparse collection. On the other hand, if $\langle \mathcal{T}_{P,K} (F\mathbf{1}_Q), G\mathbf{1}_Q\rangle\not=0$, then $Q \cap (Q+P(t)) \neq \emptyset$ for some $t$ with $|t| \ge L$. By Condition (\L) there exists $i \in \{1, \dots n\}$ such that $|P_i(t)| \ge L^{\beta}$. Thus, we see that the $i$-th side of $Q$ has length $\gtrsim L^\beta$, which implies that the sidelength of $Q$ is $\gtrsim L^{\beta/D_i} \ge L^{\beta/D}$.
	\end{proof}
	
	\section{Major arc estimate}\label{sec:major}
	\label{main term section}
	
	
	In this section we prove Proposition \ref{prop:main}. Here we adopt the key technique using Hilbert spaces of Culiuc, Kesler and Lacey \cite{CKL16}. Denote by $\mathcal{H}_k$ a complex Hilbert space of dimension $\# \mathcal{R}_k$ with vectors $v\in\mathcal{H}_k$ denoted as $v=(v_{a/q})_{a/q}=(v_{a/q})_{a/q\in\mathcal{R}_k}$ with $v_{a/q}\in\C$. 
	We claim that $\langle L^{(k)}(\nabla) f,g\rangle$ can be written as an inner-product in the Hilbert space $\ell^2(\Z^n; \mathcal{H}_k)$ of the form appearing on the left-hand side of \eqref{eqn:realsparse}.
	To see this, write
	\begin{eqnarray*}
		\langle L^{(k)}(\nabla)f,g \rangle_{\ell^2(\Z^n)}%
		&=& \sum_{a/q\in\mathcal{R}_k} S(a/q)  \left\langle \Phi_{\ge k/\delta'} (\xi-a/q) \chi_k(\xi-a/q)\widehat{f}(\xi), \widehat{g}(\xi) \right\rangle_{L^2(\xi\in[0,1]^n)}\\
		&=& \sum_{a/q\in\mathcal{R}_k} \left\langle \Phi_{\ge k/\delta'} (\xi) \chi_k(\xi)S(a/q)\widehat{f}(\xi+a/q), \widehat{g}(\xi+a/q)\right\rangle_{L^2(\xi\in[0,1]^n)},
	\end{eqnarray*}
	where $\Phi_{\ge k/\delta'}=\sum_{j \ge k/\delta'} \Phi_j.$ 
	For a function $f:\Z^n\to\C$, $k\ge 1$, $a/q\in\mathcal{R}_k$ and $x\in\Z^n$ define
	\[ \mathcal{A}_{k,a/q} [f] (x) =\int_{[0,1]^n}\chi_k(\xi)\widehat{f}(\xi+a/q) e(\xi\cdot x) d\xi. \]
	We also write
	\[ \mathcal{A}_k[f](x) = (\mathcal{A}_{k,a/q}[f](x))_{a/q} \in \mathcal{H}_k.\]
	
	Define functions $F^\sharp,G:\Z^n\to \mathcal{H}_k$ by
	\[ F^\sharp(x) = (S(a/q) \mathcal{A}_{k,a/q}[f](x))_{a/q}\quad\text{and}\quad G(x) = \mathcal{A}_{k-1}[g](x). \]
	Here we chose the notation $F^\sharp$ to remind us that $S(a/q)$ appears in its definition. Since $\chi_{k-1}$ is equal to one on the support of $\chi_k$, 
	\begin{eqnarray*}
		\langle L^{(k)}(\nabla)f,g\rangle%
		&=&\sum_{a/q\in\mathcal{R}_k} \left\langle \left(\Phi_{\ge k/\delta'} \right)(\nabla) F^\sharp_{a/q}, G_{a/q} \right\rangle_{\ell^2(\Z^n)},
	\end{eqnarray*}
	which we recognize as equal to
	\[ \left\langle \left(\Phi_{\ge k/\delta'} \right)(\nabla) F^\sharp, G\right\rangle_{\ell^2(\Z^n; \mathcal{H}_k)}. \]

	Let $(\frac1r,\frac1s)\in \Omega_c$. Then Proposition \ref{prop:Hilbertsparse} (for $K=\sum_{j\ge k/\delta'} K_j$, the operator $\Phi_{\ge k/\delta'}(\nabla)$ coincides with $\mathcal{T}_{P,K}$) yields a sparse collection ${\mathcal{S}}$ (depending on $f,g$ and $k$) such that
	\begin{equation} \label{20190808eq01}
	|\langle L^{(k)}(\nabla) f,g\rangle| \lesssim \Lambda^{{\mathcal{S}}}_{r,s}(F^\sharp,G).
	\end{equation} 
	Let us make $\delta'$ small enough so that $2^{1/\delta'}\ge L_0$ (where $L_0$ is the constant from Condition (\L)). Then by Proposition \ref{prop:Hilbertsparse}, we may assume that the $P$-cubes in ${\mathcal{S}}_k$ have sidelength $\gtrsim 2^{k\gamma_*}$, where $\gamma_*=\tfrac{\beta}{D \delta'}>0$. By making $\delta'$ small enough we ensure that $\gamma_*\ge 11$.
	
	It remains to estimate $\Lambda_{r,s}^\mathcal{S}(F^\sharp, G)$ in terms of a sparse form acting on $f$ and $g$. We begin by extracting the decay of $S(a/q)$. By \eqref{20190805eq04} we have for every $\varepsilon>0$,
	\begin{equation}\label{eqn:majorarcpf_first}
	\Lambda^{{\mathcal{S}}}_{r,s}(F^\sharp,G) \lesssim_\varepsilon 2^{-k/D_*+\varepsilon} \Lambda^{{\mathcal{S}}}_{r,s}(\mathcal{A}_k[f], \mathcal{A}_{k-1}[g] ).
	\end{equation}
	
	Next we need to compare $\langle \mathcal{A}_k[f] \rangle_{Q,r}$ to $\langle f\rangle_{Q,r}$. By definition,
	\[ \langle \mathcal{A}_k[f]\rangle_{Q,r}^r = |Q|^{-1} \sum_{x\in Q} \big( \sum_{a/q\in\mathcal{R}_k} |\mathcal{A}_{k,a/q}[f](x)|^2 \big)^{r/2}. \] 
	Observe that $\mathcal{A}_{k,a/q}[f](x)$ arises by convolution of $f$ with an $\ell^1$-normalized bump function adapted to a ball of radius $2^{10k}$ centered at the origin. Therefore, if the sidelength of $Q$ is at least $2^{10k}$, we expect $|\mathcal{A}_{k,a/q}[f](x)|$ to be negligible if $f$ is supported far from $Q$. This motivates us to write
	\begin{equation}\label{eqn:majorarcpf_1}
	\mathcal{A}_k [f] = \mathcal{A}_k [f\mathbf{1}_{2Q}] + \mathcal{A}_k [f\cdot (1-\mathbf{1}_{2Q})].
	\end{equation}
	
	The first term will be controlled by orthogonality among different $a/q\in\mathcal{R}_k$. More precisely, we have the following claim.
	
	{\em Claim 1.} For every $P$-cube $Q$, $k\ge 1$ and $r\in [1,\infty]$,
	\begin{equation}\label{eqn:majorarcpf_close}
	\langle \mathcal{A}_k[f \mathbf{1}_{2Q}]\rangle_{Q,r} \lesssim 2^{\frac{n+1}{2}|\frac1r-\frac12| k} \langle f\rangle_{2Q,r}.
	\end{equation}
	
	The important feature of this estimate is that we can make the exponent arbitrarily small by having $r$ close enough to $2$.
	
	\begin{proof}[Proof of Claim 1.] We interpolate between $r=1,2,\infty$. For $r=1$ we estimate
		\[ \langle \mathcal{A}_k[f\mathbf{1}_{2Q}]\rangle_{Q,1} \le (\#\mathcal{R}_k)^{1/2} |Q|^{-1} \sum_{x\in Q} (|\mathcal{F}^{-1}(\chi_k)|*|f\mathbf{1}_{2Q}|)(x). \] 
		Since $\# \mathcal{R}^\flat_k\lesssim 2^{(n+1)k}$ and $\|\mathcal{F}^{-1}(\chi_k)\|_{\ell^1}\approx 1$, the previous is $\lesssim 2^{\frac{n+1}2 k} \langle f\rangle_{2Q,1}.$
		Similarly, 
		\[\langle \mathcal{A}_k[f\mathbf{1}_{2Q}]\rangle_{Q,\infty} \lesssim (\#\mathcal{R}_k)^{1/2} \sup_{x\in Q} (|\mathcal{F}^{-1}(\chi_k)|*|f\mathbf{1}_{2Q}|)(x)\lesssim 2^{\frac{n+1}{2}k} \langle f\rangle_{2Q,\infty}. \]
		For $r=2$ we commute $L^2$--norms, estimate $\mathbf{1}_Q\le 1$ and use Plancherel's theorem to see
		\[ \langle \mathcal{A}_k[f\mathbf{1}_{2Q}]\rangle_{Q,2}^2 \le |Q|^{-1} \sum_{a/q\in\mathcal{R}_k} \int_{[0,1]^n} |\widehat{\mathcal{A}_{k,a/q}[f\mathbf{1}_{2Q}]}(\xi)|^2 d\xi. \]
		The right-hand side is equal to
		\[ |Q|^{-1} \int_{[0,1]^n} \Big(\sum_{a/q\in\mathcal{R}_k} |\chi(2^{10k}(\xi-a/q))|^2 \Big) |\widehat{f\mathbf{1}_{2Q}}(\xi)|^2 d\xi. \]
		Using that the functions $\xi\mapsto |\chi(2^{10k}(\xi-a/q))|^2$ have disjoint support and applying Plancherel's theorem again, we see that the previous display is $\lesssim \langle f\rangle_{2Q,2}^2$.
	\end{proof}
	
	As hinted above, the second term in \eqref{eqn:majorarcpf_1} will be controlled by using rapid decay of $\mathcal{F}^{-1}(\chi_k)$, as long as the sidelength of $Q$ is large enough.
	
	{\em Claim 2.} For every $P$-cube $Q$ with $\ell(Q)\ge 2^{11k}$ and every $k\ge 1$, $r\in [1,\infty]$ and every $N\ge 1$,
	\begin{equation}\label{eqn:majorarcpf_far}
	\langle \mathcal{A}_k[f(1-\mathbf{1}_{2Q})] \rangle_{Q,r} \lesssim_N  |Q|^{-N} 2^{-kN} \sum_{\nu=1}^\infty 2^{-N\nu} \langle f\rangle_{2^\nu Q,1}.
	\end{equation}
	
	\begin{proof}[Proof of Claim 2.]
		Write $f^\sharp = f\cdot (1-\mathbf{1}_{2Q})$. Begin with the estimate
		\[ \langle \mathcal{A}_k[f^\sharp]\rangle_{Q,r} \lesssim 2^{\frac{n+1}{2} k} \big(|Q|^{-1} \sum_{x\in Q} (|\mathcal{F}^{-1}(\chi_k)|*|f^\sharp|)(x)^r \big)^{1/r}, \]
		where we have used $\# \mathcal{R}_k\lesssim 2^{(n+1)k}$. Since $\chi$ is smooth, $\mathcal{F}^{-1}(\chi_k)$ decays rapidly: for every $N\ge 1$,
		\[ |\mathcal{F}^{-1}(\chi_k)(y)| \lesssim_N 2^{-10k} (1+2^{-10k}|y|)^{-N}.\]
		If $x\in Q$ and $y\not\in 2^{\nu} Q$ for $\nu\ge 1$, then $|x-y|\gtrsim 2^\nu \ell(Q)\ge 2^{11k+\nu}$. Hence for $x\in Q$ and every $M\ge 1$,
		\[  (|\mathcal{F}^{-1}(\chi_k)|*|f^\sharp|)(x) \lesssim_M   \sum_{\nu=2}^\infty \sum_{y\in 2^\nu Q\setminus 2^{\nu-1} Q} |f(y)|(2^{-10k} \ell(Q) 2^\nu)^{-M}. \]
		Using this for large enough $M$ and using $\ell(Q)\ge 2^{11k}$ we get
		\[ \lesssim_N |Q|^{-N} \sum_{\nu=1}^\infty 2^{-(k+\nu)N} |2^\nu Q|^{-1} \sum_{y\in 2^\nu Q} |f(y)|, \]
		for every $N\ge 1$. This concludes the proof.
	\end{proof}
	
	Decomposing both $f=f\mathbf{1}_{2Q}+f\cdot(1-\mathbf{1}_{2Q})$ and $g=g\mathbf{1}_{2Q}+g\cdot(1-\mathbf{1}_{2Q})$, Claims 1 and 2 yield
	\[ \Lambda_{r,s}^{\mathcal{S}} ( \mathcal{A}_k [f], \mathcal{A}_{k-1}[g] ) \lesssim 2^{\frac{n+1}2(|\frac1r-\frac12|+|\frac1s-\frac12|)k} \sum_{Q\in \mathcal{S}} |Q| \langle f\rangle_{2Q,r} \langle g\rangle_{2Q,s} + \mathrm{remainder},\]
	where the remainder term is dominated by
	\[ \lesssim 2^{-k} \sum_{\nu=1}^\infty 2^{-100 \nu} \sum_{Q\in\mathcal{S}} |Q| \langle f\rangle_{2^\nu Q,1} \langle g\rangle_{2^\nu Q,1}. \]
	This remainder includes the mixed terms stemming from interaction of $f\mathbf{1}_{2Q}$ with $g\cdot(1-\mathbf{1}_{2Q})$ et cetera. Here we have exploited the gain in $|Q|$ from Claim 2 to allow crudely estimating $r$-averages by $1$-averages: $\langle f\rangle_{Q,r}\le |Q|^{1-\frac1r} \langle f\rangle_{Q,1}$.

	If $\mathcal{S}$ is $\sigma$--sparse, then $2^\nu \mathcal{S}=\{ 2^\nu Q\,:\,Q\in\mathcal{S}\}$ is $\approx \sigma 2^{-\nu}$--sparse and can be written as a union of $\approx 2^\nu$ many $\sigma$--sparse families.  Due to our definition of sparse forms using $P$-cubes, this fact follows from comparing sidelengths of the cubes in $2^\nu \mathcal{S}$ to $\mathcal{S}$, applying the Carleson condition, and decomposing into sparse subcollections.  Here one could have instead applied Mei's lemma to contain the $2^\nu Q$ in dyadic cubes (which are the $P$-cubes on a space of homogeneous type) of roughly the same size, as is often done in the literature (see \cite{LN18} for more details). By the universal sparse domination (Lemma \ref{lem:uniformsparse}) we then have
	\[ \text{remainder} \lesssim 2^{-k} \Lambda^{*}_{1,1}(f,g). \]
	Summarizing, we  proved Proposition \ref{prop:main} for all $(\frac1r,\frac1s)\in \Omega_c$ (defined in \S \ref{sec:realvar}) with
	\begin{equation}\label{eqn:major_condition}
	\tfrac1{D_*} > \tfrac{n+1}2 (|\tfrac1r-\tfrac12|+|\tfrac1s-\tfrac12|).
	\end{equation}
	
	\section{Minor arc estimate}\label{sec:minor}
	\label{error section}
	
	The proof of Proposition \ref{prop:err} relies on the following key estimate.
	
	\begin{lem}\label{lem:error}
		There exists $\varepsilon'>0$ such that $|E_j(\xi)|\lesssim 2^{-\varepsilon' j}$ for every $j$ and $\xi$.
	\end{lem}
	
	\begin{proof} 
		We first record a preliminary observation.
		
		{\em Claim:} If $\xi\not\in \mathfrak{M_j}(a/q)$ then 
		\begin{equation}\label{eqn:pferror_vdc}
		|\Phi_j(\xi-a/q)| \lesssim 2^{-\frac{1-\delta}{D} j}.
		\end{equation}
		\begin{proof}[Proof of claim.]
			By assumption there exists $i=1,\dots,n$ such that
			\[ |\xi_i - a_i/q| \ge 2^{-(D_i-1)j} 2^{-\delta j}. \]
			Let us write $\eta = \xi-a/q$ and look at the oscillatory integral
			\[ \Phi_j(\eta) = \int_{\R^d} e(P(2^j t)\cdot \eta) \psi(t) 2^{jd} K(2^j t) dt. \]
			Let $\alpha^{(i)}$ be the distinguished multiindex from Condition (C). Then
			\[ P(2^j t)\cdot \eta = \eta_i 2^{j D_i} t^{\alpha^{(i)}} + \text{remaining terms} \]
			Thus the coefficient of $t^{\alpha^{(i)}}$ is $\ge 2^{(1-\delta)j}$, which by a well-known oscillatory integral estimate (see \cite[Ch. VIII.2, Prop. 5]{Ste93}) implies \eqref{eqn:pferror_vdc}.
		\end{proof}
		
		\noindent The proof of Lemma \ref{lem:error} splits into two cases.
		
		{\em Case I:} $\xi\in\mathfrak{M}_j$. Then there exists $a_*/q_*\in\mathcal{R}_{k_*}$ with $k_*\le j\delta'$ such that $\xi\in\mathfrak{M}_j(a_*/q_*)$. We have
		\[ L_{j,k_0}(\xi) = S(a_*/q_*) \Phi_j(\xi-a_*/q_*) \]
		and by Lemma \ref{lem:approx} this is $m_j(\xi)+O(2^{-j\varepsilon})$. Thus
		\begin{equation}\label{eqn:errorestpf0}
		|E_j(\xi)| \le \sum_{\substack{1\le k\le j\delta',\\k\not=k_*}} |L_{j,k}(\xi)|
		\end{equation}
		For each $k$ there exists at most one $a/q\in\mathcal{R}_k$ such that
		$\chi_k(\xi-a/q)\not=0$. If $k\not=k_*$ and $k\le j\delta'$ then we know additionally that $\xi\not\in\mathfrak{M}_j(a/q)$ because the major arcs in $\mathfrak{M}_j$ are disjoint. Then \eqref{eqn:errorestpf0} and \eqref{eqn:pferror_vdc} give
		\[ |E_j(\xi)| \lesssim j 2^{-\frac{1-\delta}{D} j}. \]
		
		{\em Case II:} $\xi\not\in\mathfrak{M}_j$. We begin with the estimate		
		\[ |E_j(\xi)|\le |L_j(\xi)| + |m_j(\xi)|. \]
		Since $\xi\not\in\mathfrak{M}_j(a/q)$ and for each $k$ there exists at most one $a/q\in\mathcal{R}_k$ with $\chi_k(\xi-a/q)\not=0$, we have  $|L_j(\xi)|\lesssim j 2^{-\frac{1-\delta}{D}j}$ by \eqref{eqn:pferror_vdc}.
		To estimate $|m_j(\xi)|$ we use \cite[Proposition 3]{SW99}.
		Let $\epsilon_0=\delta'/n$. By Dirichlet's approximation theorem there exist reduced fractions $a_1/q_1,\dots,a_n/q_n\in\mathbb{Q}$ with $(a_l,q_l)=1$ and $q_l\le 2^{(D_l - \epsilon_0)j}$ and
		\[ |\xi_l-a_l/q_l| \le \tfrac1{q_l} 2^{-(D_l-\epsilon_0)j}\]
		for every $l=1,\dots,n$. Since $\epsilon_0< \frac{\delta}{10n} <1-\delta$ we have $\xi\in \mathfrak{M}_j( (a_l/q_l)_l )$. But also $\xi\not\in\mathfrak{M}_j$, so the least common multiple of $q_1,\dots,q_n$ must be greater than $2^{j\delta'}$. Hence there must exist $i=1,\dots,n$ such that $q_i\ge 2^{j\delta'/n}=2^{j\epsilon_0}$. From Condition (C) we have that 
		\[ P(y)\cdot \xi = \xi_i y^{\alpha^{(i)}} + \text{remaining terms} \]
		Thus \cite[Proposition 3]{SW99} implies that there exists some $\varepsilon'>0$ such that $|m_j(\xi)|\lesssim 2^{-\varepsilon' j}$ as desired.		 		
	\end{proof}
	
	In addition we note the crude derivative estimate
	\[
	|\partial_{\alpha} E_j(\xi)| \lesssim_\alpha 2^{(\alpha\cdot\underline{D}) j},
	\]
	valid for all multiindices $\alpha$, where $\underline{D}=(D_1,\dots,D_n)$.
	Let us write $\mathcal{K}_j = \mathcal{F}^{-1}[E_j]$.
	
	We define the anisotropic norm $\rho(x)=\max_{i} |x_i|^{1/D_i}$.
	Integrating by parts $N$ times in the coordinate $i_0$ so that $\rho(x)=|x_{i_0}|^{1/D_{i_0}}$ we get for $|x|\ge 1$,
	\begin{equation}\label{eqn:err_kernelest1}
	|\mathcal{K}_j (x)| \lesssim_N 2^{ N D_{i_0} j} \frac{1}{|x_{i_0}|^N} \le 2^{N D j} \frac1{\rho(x)^N}
	\end{equation}
	for all positive integers $N$, where we have estimated $|x_{i_0}|\ge |x_{i_0}|^{1/D_{i_0}}= \rho(x)$. On the other hand, Lemma \ref{lem:error} tells us that
	\begin{equation}\label{eqn:err_kernelest2}
	|\mathcal{K}_j (x)| \lesssim 2^{-\varepsilon' j}.
	\end{equation}
	
	The idea is to apply Proposition \ref{prop:Kfinitesupport}. To do this we need to restrict $\mathcal{K}_j$ to some large enough $P$-cube $Q_*$. Specifically, we want $Q_*$ large enough so that the $\ell^1$ sum of $\mathcal{K}_j$ on the complement of $Q_*$ is small. Motivated by \eqref{eqn:err_kernelest1}, we set
	\[ 	Q_* = \{ x\in\Z^n\,:\,\rho(x)\le 2^{(D+1) j} \}. \]
	Then applying \eqref{eqn:err_kernelest1} with $N$ large and using $|\{x\in\Z^n\,:\,\rho(x)\le L\}|\approx L^{|\underline{D}|}$ (where $|\underline{D}|=\sum_{i=1}^n D_i$),
	\begin{equation}\label{eqn:err_kernelL1}
	\|\mathcal{K}_j \mathbf{1}_{Q_*^c}\|_{\ell^1} \lesssim_N 2^{N D j} \sum_{k\ge 0} \sum_{\rho(x)\approx 2^{(D+1)j+k}} \rho(x)^{-N} \approx 2^{(|\underline{D}|(D+1)-N)j} \sum_{k\ge 0} 2^{-k(N-|\underline{D}|)},
	\end{equation}
	which implies that
	\[ \|\mathcal{K}_j \mathbf{1}_{Q_*^c}\|_{\ell^1} \lesssim_M 2^{-jM}\]
	for all $M\ge 1$. Using that the right hand side of \eqref{eqn:err_kernelest1} (with $N=|\underline{D}|(D+1)+1$) provides a radially (radial with respect to $\rho$) decreasing integrable majorant for $\mathcal{K}_j \mathbf{1}_{Q_*^c}$ with $\ell^1$ norm $\lesssim 2^{-j}$, we get
	\[ |(\mathcal{K}_j \mathbf{1}_{Q_*^c})*f| \lesssim 2^{-j} M_P f(x). \]
	By Proposition \ref{sparseHL} we obtain
	\[ \| f\mapsto \mathcal{K}_j \mathbf{1}_{Q_*^c}*f \|_{\mathrm{sp}(1,1)} \lesssim 2^{-j}. \]
	It remains to treat the kernel $\mathcal{K}_j \mathbf{1}_{Q_*}$. We begin by establishing the $\ell^p$--improving estimates for $\mathcal{K}_j \mathbf{1}_{Q_*}$. 	
	We claim that for all $r\in [1,2]$,
	\begin{equation}\label{eqn:err_kernel0_Lpimproving1}
	\|(\mathcal{K}_j \mathbf{1}_{Q_*}) * f\|_{\ell^{r'}} \lesssim 2^{-\varepsilon' j} \|f\|_{\ell^r}.
	\end{equation}
	This is by interpolation between the endpoints $r=1$ and $r=2$. The case $r=2$ follows from Plancherel's theorem and Lemma \ref{lem:error} and the case $r=1$ is a consequence of \eqref{eqn:err_kernelest2}.
	Additionally, setting $N_P=|\underline{D}|(D+1)$, we have the estimate
	\begin{equation}\label{eqn:err_kernel0_L11}
	\|(\mathcal{K}_j \mathbf{1}_{Q_*}) * f\|_{\ell^1} \lesssim 2^{(N_P-\varepsilon')j} \|f\|_{\ell^1}
	\end{equation}
	from \eqref{eqn:err_kernelest2} and definition of $Q_*$. Interpolating \eqref{eqn:err_kernel0_Lpimproving1} and \eqref{eqn:err_kernel0_L11} we obtain 
	\begin{equation}
	\| (\mathcal{K}_j \mathbf{1}_{Q_*}) * f\|_{\ell^{s'}} \lesssim 
	2^{-\varepsilon' j + N_P (\frac1r - \frac1s)j}
	\|f\|_{\ell^r} 	
	\end{equation}
	for every $(\frac1r,\frac1s)$ in the closed triangle $\Delta$ with vertices $(\tfrac12,\tfrac12), (1,0), (1,1)$ (in particular $\frac1r-\frac1s>0$).
	Proposition \ref{prop:Kfinitesupport} then yields
	\begin{equation}
	\| f\mapsto (\mathcal{K}_j \mathbf{1}_{Q_*})*f \|_{\mathrm{sp}(r,s)} \lesssim 2^{-\varepsilon' j + 2 N_P (\frac{1}r-\frac12)j}
	\end{equation}
	for all $(\frac1r, \frac1s)\in \Delta$. 
	From switching the roles of $r$ and $s$ and H\"older's inequality we then conclude that for every $(\frac1r,\frac1s)$ with $\max\{\tfrac1s, \tfrac1r\} <\tfrac12 + \tfrac{\varepsilon'}{2N_P}$ we have
	\[ \| f\mapsto (\mathcal{K}_j \mathbf{1}_{Q_*})*f \|_{\mathrm{sp}(r,s)} \lesssim 2^{-\gamma j} \]
	for some $\gamma>0$. This concludes the proof of Proposition \ref{prop:err}.
	
	\section{Quantitative sparse region}\label{sec:quant}

	Let $\Omega_m\subset [0,1]^2$ denote the set of $(\tfrac1r,\tfrac1s)~\in~[0,1]^2$ such that
	\[ \max(\tfrac1s,\tfrac1r) < \tfrac12+\tfrac{\varepsilon'}{2N_P}, \]
	where $N_P=(1+\mathrm{deg}\,P)(\sum_{i=1}^n \deg\,P_i)$ and $\varepsilon'$ is the exponent from Lemma \ref{lem:approx} (which depends among other quantities on the amount of power decay obtained in \cite[Proposition 3]{SW99}). This is the condition needed for the minor arc estimate in \S \ref{sec:minor}. Note that it implies the condition \eqref{eqn:major_condition} needed for the major arc term.
	Let $\Omega_c\subset [0,1]^2$ be the set defined in \S \ref{sec:realvar} (it is determined by the sharp $L^p$ improving region for the corresponding single-scale real-variable operator). 
	Then we have proved
	\[ (\tfrac1r,\tfrac1s)\in\Omega_m\cap \Omega_c\quad\Longrightarrow\quad \|T_P\|_{\mathrm{sp}(r,s)}<\infty. \]

\begin{center}
\begin{figure}[H]
\centering 
\begin{tikzpicture}[scale=5]
\draw (0,0) [->] -- (0,1.2) node [left] {$\tfrac{1}{s}$};
\draw (0,0) [->] -- (1.2,0) node [below] {$\tfrac{1}{r}$};
\draw [dashed] (0.85, 0.45)--(0.45, 0.85); 
\draw [dashed]  (1, 0)--(0, 1);
\draw [dashed] (0.85, 0.45)--(0.85, 0); 
\draw [dashed] (0.45, 0.85)--(0, 0.85); 
\fill [opacity=. 1] (0, 0)--(0.6, 0)--(0.6, 0.6)--(0, 0.6)--cycle; 
\draw (0, 1)  node [left] {$(0, 1)$} -- (1, 1)  -- (1, 0) node [below] {$(1, 0)$};
\fill (.265, .295) node [right] {$\Omega_m$}; 
\draw (0,0) [->] -- (0,1.2) node [left] {$\tfrac{1}{s}$};
\draw (0,0) [->] -- (1.2,0) node [below] {$\tfrac{1}{r}$};
\draw [dashed] (.6, .6) -- (.6, 0) ; 
\draw [dashed] (0, .6)  -- (.6, .6);
\draw [dashed]  (1, 0)--(0, 1);
\draw (0, 1)  node [left] {$(0, 1)$} -- (1, 1)  -- (1, 0) node [below] {$(1, 0)$};
\end{tikzpicture}
\caption{Range of $(\tfrac1r, \tfrac1s)$ for Theorem \ref{mainresult}. The larger quadrangle corresponds to the restriction imposed by the major arc estimate \eqref{eqn:major_condition}.}
\label{sparserangefigure}
\end{figure}
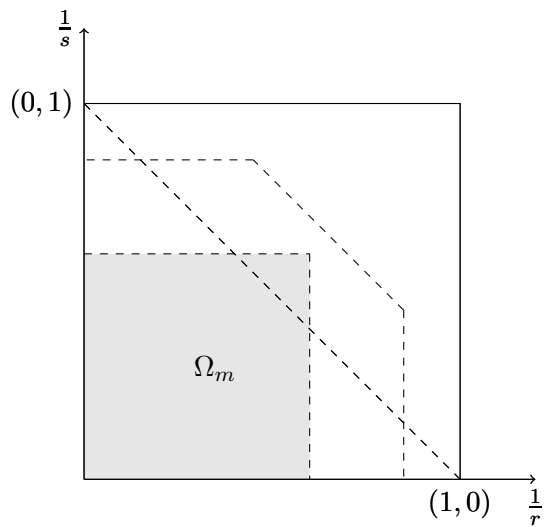
\end{center}

\end{document}